\documentclass[11pt, a4paper, twoside]{article}
\usepackage{amsmath, amsthm}
\usepackage{amsfonts}
\usepackage{amscd}
\usepackage{sectsty}
\usepackage{latexsym}
\usepackage{amssymb, hyperref}
\usepackage{amsthm}
\usepackage{multirow}
\usepackage{xspace}
\usepackage{graphicx}
\usepackage{fancyhdr}
\usepackage{blindtext}
\usepackage{titlesec}
\usepackage[affil-it]{authblk} 
\usepackage{titling}
\newtheoremstyle{theorem}
  {10pt}		  
  {10pt}  
  {\sl}  
  {\parindent}     
  {\bf}  
  {. }    
  { }    
  {}     
\theoremstyle{theorem}
\newtheorem{theorem}{Theorem}
\newtheorem{corollary}[theorem]{Corollary}

\newtheoremstyle{defi}
  {10pt}		  
  {10pt}  
  {\rm}  
  {\parindent}     
  {\bf}  
  {. }    
  { }    
  {}     
\theoremstyle{defi}

\sectionfont{\normalsize \bfseries}
\begin{document}

\title{On the Objects and Morphisms of Category of Soft sets}
{\small
\author{Ratheesh K. P.$^1$ and Sunil Jacob John$^2$\\
$^1$$^,$$^2$Department of Mathematics\\ National Institute of Technology, Calicut, India, Pin-673601\\
$^1$ratheeshmath@gmail.com\\$^2$sunil@nitc.ac.in}}

\date{}

\maketitle
\vspace{-.55in} 
\rule{\linewidth}{0.75pt}\vspace{0.05in}

\begin{abstract}
Soft set theory can deal uncertainties in nature by parametrization process. In this paper, we explore the objects and morphisms of category of soft sets, \textbf{Sset(U)} in detail. Also, gives characterizations of monomorphisms and epimorphisms in \textbf{Sset(U)}.



\rule{\linewidth}{0.75pt}

\end{abstract}

\section{Introduction}

The concept of soft sets introduced by D. Molodtsov \cite{mol}, is widely applied for solving problems in several areas including vagueness and uncertainties. In this paper, we devoted to presenting basic categorical ideas of soft sets over a fixed universal set $U$ and its properties. The class of all soft set over $U$ is denoted by \textbf{Sset(U)}.\\

In \cite{zho}, M. Zhou et al. initiated to develop some categorical properties of soft sets such as equalizers, finite products, adjoint. R. A. Barzooei et al. \cite{bor} explain the concepts of epimorphism and monomorphism using a pair of functions as a morphism. O. Zahiri \cite{zah} discussed the existence of product and coproduct in the category of soft sets. In \cite{Liu}, Z. Liu proved that the category of soft sets is a topos. B. P. Varol et al. \cite{var} investigate the soft neighborhood sets and its topological properties in detail. In \cite{sar}, S. K. Sardar et al. defines soft category as a soft set and prove that soft category is a generalization of the fuzzy category. 

The remaining portion of this article is organized as follows. In Sec. 2, some basic notions in category theory and soft set theory are given. In Sec. 3, objects in the category of soft sets are studied and establish a necessary and sufficient condition for an object to an initial object, terminal object, separator, and co-separator. In Sec 4, morphisms in the category of soft sets and establish some characterizations of epimorphism and monomorphism in detail.   

\section{Preliminaries}

The basic category theory terminologies are taken from J. Adámek, H. Herrlich and G. E. Strecker \cite{str} and we use standard terminology in soft sets from  \cite{maji}, \cite{mol}.
	
A \textbf{category} is a quadruple  $C = (\mathcal{O}, hom, id, \circ)$ consisting of
	\begin{itemize}
		\item[1.] A class $\mathcal{O}$, whose members are called $C$-objects
		\item[2.] For each pair $(A,B)$ of $C$-objects, a set $hom(A,B)$, whose members are called $C$-morphisms from $A$ to $B$
		\item[3.] For each $C$-object $A$, a morphism 
		$id_A: A \rightarrow A$, called the $C$-identity on $A$
		\item[4.] A composition law associating with each $C$-morphism $f: A\rightarrow B $ and each $C$-morphism $g: B \rightarrow C$ an $C$-morphism $g \circ f: A \rightarrow C$, called the composite of $f$ and $g$, subject to the following conditions:\\
		(a) composition is associative\\
		(b) $C$-identities act as identities with respect to composition\\
		(c) the sets $hom(A,B)$ are pairwise disjoint
	\end{itemize}

An object $A$ is said to be an \textbf{initial object} if for each object $B$, $|hom(A,B)|= 1$ and an object $A$ is called a \textbf{terminal object} provided that for each object $B$, $|hom(B,A)|=1$. An object $A$ is called a \textbf{zero object}, if it is both an initial object and a terminal object. An object $S$ is called a \textbf{separator} provided that whenever $f,g: A \rightarrow B$ are distinct morphisms, there exists a morphism $h: S \rightarrow A$ such that $f\circ h \neq g\circ h$. An object $C$ is called a \textbf{co-separator} provided that whenever $f,g: A \rightarrow B$ are distinct morphisms there exists a morphism $k: B \rightarrow C$ such that $k\circ f \neq k\circ g$.
A morphism $f: A \rightarrow B$ is said to be an \textbf{epimorphism}, if for all pairs $h,k : B \rightarrow C$ of morphisms such that $h \circ f = k \circ f$, implies $h = k$. That is, f is right cancellable for composition. A morphism $f: A \rightarrow B$ is said to be a \textbf{monomorphism}, provided that for all pairs $h,k : C \rightarrow A$ of morphisms such that $f \circ h = f \circ k$, implies $h = k$. That is, f is left cancellable for composition. A morphism is called a \textbf{bimorphism} provided that it is both a monomorphism
and an epimorphism.\\
  
Let $\mathcal{F}_A$ and $\mathcal{G}_B$ be two soft sets over $U$. A soft morphism (Soft function \cite{zho}) from $\mathcal{F}_A$ to $\mathcal{G}_B$ is a function $\alpha : A \rightarrow B$ such that $\mathcal{F} (a)\subseteq (\mathcal{G} \circ \alpha)(a)$, for each $a\in A$. The composition of two soft morphisms is again a soft morphism. Each soft morphism $\alpha:\mathcal{F}_A \rightarrow \mathcal{G}_B$ in $hom(\mathcal{F}_A, \mathcal{G}_B)$ is considered as a triple $(\mathcal{F}_A,\alpha,\mathcal{G}_B)$. The composition of soft morphisms is associative and identity function is the identity soft morphism. This establishes that \textbf{Sset(U)} is a category.

\section{Objects in \textbf{Sset(U)} }

 In this section, we characterize initial object, terminal object, separator and co-separator of the category, \textbf{Sset(U)} in detail.  

\begin{theorem}
	An object $\mathcal{F}_A \in \textbf{Sset(U)}$ is an initial object if and only if $A = \emptyset$. 
\end{theorem}
\begin{proof}
	Let $\mathcal{F}_A$ be an initial object in the category of soft sets, $\textbf{Sset(U)}$. Then $|Mor(\mathcal{F}_A, \mathcal{G}_B)| = 1$, for every soft set $\mathcal{G}_B$ in $\textbf{Sset(U)}$.
	If $A \neq \emptyset$. Then $|Mor(\mathcal{F}_A, \mathcal{G}_B)| > 1$, for every absolute soft soft set $\mathcal{G}_B$. So $A$ is a null set.\\
	
	Conversely assume $A=\emptyset$, there is exactly one function from empty set to any set. So $\mathcal{F}_A$ is an initial object in $\textbf{Sset(U)}$.
\end{proof}

\begin{theorem}
	An object $\mathcal{F}_A \in \textbf{Sset(U)}$ is a terminal object if and only if $\mathcal{F}_A$ is an absolute soft set with $|A| = 1$. 
\end{theorem}

\begin{proof}
	Let $\mathcal{F}_A$ be an absolute soft set with $|A| = 1$, then there is exactly one soft morphism (constant function) between any soft set $\mathcal{G}_B$ and $\mathcal{F}_A$.\\
	
	Conversely, if $|A| > 1$. Then $|Mor(\mathcal{G}_B, \mathcal{F}_A)| > 1$ for every null soft set $\mathcal{G}_B$, which is a contradiction. So $|A| = 1$. Also, if $\mathcal{F}_A$ is not an absolute soft set, then there is no soft morphism between absolute soft set $\mathcal{G}_B$ and $\mathcal{F}_A$. So $\mathcal{F}_A$ must be an absolute soft set with $|A| = 1$.
\end{proof}

\begin{theorem}
	The category \textbf{Sset(U)} has no zero object.
\end{theorem}

\begin{theorem}
	An object $\mathcal{H}_C$ in \textbf{Sset(U)} is a separator if and only if $\mathcal{H}$ is a null soft set with $C\neq \emptyset$. 
\end{theorem}

\begin{proof}
	Suppose $\mathcal{H}_C$ is a null soft set with $C\neq \emptyset$. Let $\mathcal{F}_A$ and $\mathcal{G}_B$ be two object in \textbf{Sset(U)} and $\alpha, \beta : \mathcal{F}_A \rightarrow \mathcal{G}_B$ be two distinct soft morphisms, ie $\alpha(a) \neq \beta(a)$, for some $a\in A$.\\
	Define $\gamma: C\rightarrow A$ by $\gamma(c) = a$, for all $c\in C$. By definition of $\mathcal{H}$, $\gamma$ is a soft morphism from $\mathcal{H}_C$ to $\mathcal{F}_A$ (ie $\emptyset = \mathcal{H}(c) \subseteq (\mathcal{F} \circ \gamma)(c)$, for all $c\in C$). Also $(\alpha \circ \gamma)(c)= \alpha(a)$ and $(\beta \circ \gamma)(c)= \beta(a)$. So $\alpha \circ \gamma \neq \beta \circ \gamma$. Hence $\mathcal{H}_C$ is a separator in \textbf{Sset(U)}.\\
	
	Conversely assume that $\mathcal{H}_C$ is a separator in \textbf{Sset(U)}. If $C= \emptyset$, then $\alpha \circ \gamma = \beta \circ \gamma$. Now assume there exists  $c \in C$ such that $\mathcal{H}(c) \neq \emptyset$, then there is no soft morphism between $\mathcal{H}_C$ and null soft set $\mathcal{F}_A$. Hence $\mathcal{H}_C$ is a null soft set with $C \neq \emptyset$.
\end{proof}

\begin{theorem}
	An object $\mathcal{H}_C$ in \textbf{Sset(U)} is a co-separator if and only if there exist $c_1, c_2 \in C$, $c_1 \neq c_2$ such that  $\mathcal{H}(c_1)=\mathcal{H}(c_2) = U$. 
\end{theorem}

\begin{proof}
	Suppose $c_1, c_2 \in C$, $c_1 \neq c_2$ with $\mathcal{H}(c_1)=\mathcal{H}(c_2) = U$.  Let $\alpha, \beta : \mathcal{F}_A \rightarrow \mathcal{G}_B$ be two distinct soft morphisms, i.e. $\alpha(a_1) \neq \beta(a_1)$, for some $a_1\in A$. \\
	Define  $\gamma: B\rightarrow C$ by $ \gamma(b)  = \left\{ \begin{array}{ll}
	c_1 & \mbox{if $ b=\alpha(a_1) $};\\
	c_2 & \mbox{$ otherwise $}\end{array} \right.$\\ Then $\gamma: \mathcal{G}_B \rightarrow \mathcal{H}_C$ is a soft morphism, since $\mathcal{G}(b) \subseteq U= \mathcal{H}(c_1)=\mathcal{H}(c_2)$. Also, $(\gamma \circ \alpha)(a_1) = c_1 \neq c_2= (\gamma \circ \beta)(a_1)$. Hence $\mathcal{H}_C$ is a co-separator.\\
	
	Conversely, if $C= \emptyset$, then there is no soft morphism between $\mathcal{G}_B$ and $\mathcal{H}_C$ and if $|C|=1$, then $\gamma \circ \alpha= \gamma \circ \beta$. Therefore $C$ contains at least two distinct points. Since $\mathcal{H}_C$ is a co-separator, then there exists a soft morphism $\gamma : \mathcal{G}_B \rightarrow \mathcal{H}_C$ such that  $\gamma \circ \alpha \neq \gamma \circ \beta$, i.e. $(\gamma \circ \alpha)(a) = c_1 \neq c_2= (\gamma \circ \beta)(a)$, for some $a\in A$.
	Now let $\alpha(a)= b_1$ and $\beta(a)= b_2$ and assume that $\mathcal{G}$ is an absolute soft set, then $\mathcal{H}(c_1) = (\mathcal{H}\circ \gamma)(b_1) = U$ and $\mathcal{H}(c_2) = (\mathcal{H}\circ \gamma)(b_2) = U$.	
\end{proof}

\section{Morphisms in \textbf{Sset(U)} }
In this section, we investigate the notion of epimorphism and mono morphism and characterize them.

\begin{theorem}\label{epi}
	A soft morphism $\alpha: \mathcal{F}_A \rightarrow \mathcal{G}_B $ is an epimorphism  if and only if $\alpha$ is surjective.
\end{theorem}
\begin{proof}
	Assume that $\alpha: \mathcal{F}_A \rightarrow \mathcal{G}_B $ is an epimorphism, 
	Let $C= \{0,1\}$  and $\mathcal{H}_C$ be an absolute soft set. Define $\beta, \gamma: B \rightarrow C$ by $\beta(B) =\{0\}$ and $ \gamma(b)  = \left\{ \begin{array}{ll}
	0 & \mbox{if $ b=\alpha(a),  a\in A $};\\
	1 & \mbox{$ otherwise $}\end{array} \right.$\\
	Then $\beta, \gamma : \mathcal{G}_B \rightarrow \mathcal{H}_C$  are soft morphisms with $\beta \circ \alpha = \gamma \circ \alpha $, which implies  $\beta= \gamma$, since $\alpha$ is an epimorphism. From the definition of $\beta$ and $\gamma$, $\beta(B)=\{0\}= \gamma(\alpha(A))$. Hence $\alpha(A)= B$, i.e. $\alpha$ is surjective.\\
	
	Conversely assume that $\alpha$ is surjective. Let $\beta, \gamma : \mathcal{G}_B \rightarrow \mathcal{H}_C$ be soft morphisms with $\beta \circ \alpha = \gamma \circ \alpha$. Let $b\in B$, then there exists $a\in A$ such that $\alpha(a)= b$.\\
	Now, for any $b\in B$
	\begin{eqnarray*}
		\beta(b) & = & \beta(\alpha(a)) \\
		& = & (\beta \circ \alpha)(a) \\
		& = & (\gamma \circ \alpha)(a)  \\
		& = & \gamma(\alpha(a))\\ 
		& = & \gamma(b) 
	\end{eqnarray*}	
	Which implies $\beta = \gamma$. Hence $\alpha: \mathcal{F}_A \rightarrow \mathcal{G}_B $ is an epimorphism.
\end{proof}

\begin{theorem}\label{mono}
	A soft morphism $\alpha: \mathcal{F}_A \rightarrow \mathcal{G}_B $ is a monomorphism  if and only if $\alpha$ is injective.
\end{theorem}

\begin{proof}
	Suppose $\alpha $ is a monomorphism. Let $a_1, a_2 \in A$ and let $C=\{c\}$ and $\mathcal{H}(c)= \emptyset$. Define $\beta, \gamma : C \rightarrow A$ by $\beta(c)= a_1$ and $\gamma(c)= a_2$. By definition of $\mathcal{H}$, $\beta$ and $ \gamma$ are soft morphisms from $\mathcal{H}_C$ to $\mathcal{F}_A$. For $a_1, a_2 \in A$, 
	\begin{eqnarray*}
		\alpha(a_1) = \alpha(a_2) & \implies & (\alpha \circ \beta)(c) = (\alpha \circ \gamma)(c) \\
		& \implies & \beta(c)= \gamma(c)\hspace{0.8cm} \text{Since $\alpha$ is a monomorphism} \\
		& \implies & a_1= a_2 	
	\end{eqnarray*}
	Which implies $\alpha$ is injective. \\
	
	Conversely assume $\alpha$ is injective. Let $\beta, \gamma : \mathcal{H}_C \rightarrow \mathcal{F}_A$ be soft morphisms with $\alpha \circ \beta = \alpha \circ \gamma$. Now, for any $c\in C$
	\begin{eqnarray*}
		(\alpha \circ \beta)(c) = (\alpha \circ \gamma)(c) & \implies & \alpha ( \beta(c)) = \alpha ( \gamma(c)) \\
		& \implies & \beta(c)= \gamma(c)\hspace{0.8cm} \text{Since $\alpha$ is injective} \\
		& \implies & \beta = \gamma	
	\end{eqnarray*}
	Hence $\alpha$ is a monomorphism.
\end{proof}

\begin{corollary}
	In \textbf{Sset(U)}, the bimorphisms are precisely the bijective soft morphisms.
\end{corollary}

\begin{proof}
	Use the definition of bimorphism and Theorem \ref{epi} and  \ref{mono}.
\end{proof}
\begin{theorem}
	A soft morphism $\alpha : \mathcal{F}_A \rightarrow \mathcal{G}_B$ is an isomorphism if and only if $\alpha$ is bijective and $\mathcal{F}(a)= (\mathcal{G}\circ \alpha)(a)$, for all $a\in A$.
\end{theorem}

\begin{proof}
	Suppose $\alpha$ is an isomorphism, then there exists a soft morphism $\beta : \mathcal{G}_B \rightarrow \mathcal{F}_A$ such that $\beta \circ \alpha = id_{\mathcal{F}_A}$ and $\alpha \circ \beta = id_{\mathcal{G}_B}$, i.e. $\beta \circ \alpha = id_A$ and $\alpha \circ \beta = id_B$ are identity functions. So $\alpha$ is bijective. Let $a\in A$ and $\alpha(a)= b$, then $\beta(b)=a$. Since $\alpha$ and $\beta$ are soft morphisms, we have 
	\begin{center}
		$\mathcal{F}(a) \subseteq (\mathcal{G}\circ \alpha)(a)= \mathcal{G}(b) \subseteq (\mathcal{F}\circ \beta)(b)= \mathcal{F}(a)$
	\end{center}	
	Hence $\mathcal{F}(a) = (\mathcal{G}\circ \alpha)(a)$, for all $a \in A$.\\
	
	Conversely assume that $\alpha$ is bijective and $\mathcal{F}(a)= (\mathcal{G}\circ \alpha)(a)$, for all $a\in A$. Then there exist a function $\beta: B \rightarrow A$ such that $\beta \circ \alpha = id_A$ and $\alpha \circ \beta = id_B$. So $\beta \circ \alpha = id_{\mathcal{F}_A}$ and $\alpha \circ \beta = id_{\mathcal{G}_B}$ are identity soft morphisms. Hence $\alpha$ is an isomorphism.
\end{proof}

\section{Conclusions}
In this paper,  we have investigated objects and morphisms in the category of soft sets, \textbf{Sset(U)}. All these concepts are basic supporting structures for research and development on the categorical version of the soft set theory. For future work, one can study the sections, retractions, regular and extremal monomorphism, regular and extremal epimorphism and quotient objects in the category of soft sets.

\section{Acknowledgment}
The first author is very much indebted to University Grants Commission, India for awarding Teacher Fellowship under Faculty Development Programme ($XII^{th}$ plan).


\begin{thebibliography}{99}

\bibitem{str} J. Adámek, H. Herrlich and G. E. Strecker, {\it Abstract and concrete categories. The joy of cats},  (2004).

\bibitem{bor} R. A. Borzooei, M. Mobini, and M. M. Ebrahimi. {\it The category of soft sets},  Journal of Intelligent and  Fuzzy Systems {\bf 28.1} (2015), 157-167.

\bibitem {maji} P. K. Maji, R. Biswas and A. R. Roy, {\it Soft set theory}, Computers and Math. Appl. {\bf 45}(2003) 555-562.

\bibitem{mol} D. Molodtsov, {\it Soft Set Theory- First Results}, Computers Math. Appl. {\bf 37} (1999), 19-31.

\bibitem{Liu} Z. Liu, X. Yuan, Y. Zhang, and Y. Wang. {\it Study of soft sets category and it's properties}, In Cloud Computing and Intelligent Systems (CCIS), IEEE 2nd International Conference, {\bf 3} (2012), 1377-1380.

\bibitem{sar} S. K. Sardar, and S. Gupta. {\it Soft Category Theory-An Introduction.} Journal of Hyperstructures {\bf 2.2} (2013), 118-135.

\bibitem{var} B. P. Varol, and H. Ayg$\ddot{\text{u}}$n {\it Soft sets over power sets: Generalities and applications to topology },  Journal of Intelligent and  Fuzzy Systems {\bf 29} (2015), 389-395.

\bibitem{zah} O. Zahiri, {\it Category of soft sets} Annals of the University of Craiova-Mathematics and Computer Science Series {\bf 40.2} (2013), 154-166.

\bibitem{zho} M. Zhou, L.Shenggang and M. Akram, {\it Categorical properties of soft sets}, The scientific World Journal (2014), 1-9.

\end{thebibliography}
\end{document}